\documentclass{article}
\usepackage{amssymb}
\usepackage{amsmath}
\usepackage{amsthm}
\usepackage{color}
\usepackage[dvipsnames]{xcolor}
\usepackage{hyperref}
\usepackage{tikz}

\newtheorem{theorem}{Theorem}[section]
\newtheorem{lemma}[theorem]{Lemma}
\newtheorem{proposition}[theorem]{Proposition}
\newtheorem{corollary}[theorem]{Corollary}

\theoremstyle{definition}

\newtheorem{example}[theorem]{Example}
\newtheorem{algorithm}[theorem]{Algorithm}

\theoremstyle{remark}

\numberwithin{equation}{section}

\begin{document}

\title{On the enumeration of the set of atomic numerical semigroups with fixed Frobenius number}

\author{Aureliano M. Robles-P\'erez\thanks{Both authors are supported by the project MTM2014-55367-P, which is funded by Mi\-nis\-terio de Econom\'{\i}a y Competitividad and Fondo Europeo de Desarrollo Regional FEDER, and by the Junta de Andaluc\'{\i}a Grant Number FQM-343. The second author is also partially supported by the Junta de Andaluc\'{\i}a/Feder Grant Number FQM-5849.} \thanks{Departamento de Matem\'atica Aplicada, Universidad de Granada, 18071-Granada, Spain. \newline E-mail: {\bf arobles@ugr.es}}
	\mbox{ and} Jos\'e Carlos Rosales$^*$\thanks{Departamento de \'Algebra, Universidad de Granada, 18071-Granada, Spain. \newline E-mail: {\bf jrosales@ugr.es}} }

\date{ }

\maketitle

\begin{abstract}
	A numerical semigroup is irreducible if it cannot be obtained as intersection of two numerical semigroups containing it properly. If we only consider numerical semigroups with the same Frobenius number, that concept is generalized to atomic numerical semigroup. Based on a previous one developed to obtain all irreducible numerical semigroups with a fixed Frobenius number, we present an algorithm to compute all atomic numerical semigroups with a fixed Frobenius number.
\end{abstract}

\medskip

\noindent{\it 2010 AMS Classification:} 20M14, 11Y16.

\section{Introduction}

Let $\mathbb{Z}$ and $\mathbb{N}$ be the set of integers and non-negative integers, respectively. A \emph{numerical semigroup} is a subset $S$ of $\mathbb{N}$ that is closed under addition, $0 \in S$, and $\mathbb{N} \setminus S$ is finite. We denote by $\mathcal{S}$ the set formed by all numerical semigroups. 

The greatest integer that does not belong to $S$ is the so-called \emph{Frobenius number of $S$} and it is denoted by $\mathrm{F}(S)$. If $F$ is a positive integer, then we denote by $\mathcal{S}(F) = \left\{ S \in \mathcal{S} \mid \mathrm{F}(S)=F \right\}$.

In \cite{pacific} it was introduced the notion of \emph{irreducible numerical semigroup}: a numerical semigroup is irreducible if it cannot be expressed as intersection of two numerical semigroups containing it properly. In \cite{pacific} it was also proven that a numerical semigroup $S$ is irreducible if and only if $S$ is maximal (with respect to the inclusion order) in the set $\mathcal{S}\big(\mathrm{F}(S)\big)$. Moreover, from \cite{barucci} (respectively, \cite{froberg}), we deduce that the class of irreducible numerical semigroups with even (respectively, odd) Frobenius number is just the class of pseudo-symmetric (respectively, symmetric) numerical semigroups. Let us recall that the motivation of symmetric numerical semigroups was the classical Kunz's paper (\cite{kunz}), where it was shown that a one-dimensional analytically irreducible Noetherian local ring is Gorenstein if and only if its value semigroup is symmetric. As a generalization of this result, in \cite{barucci} was shown that a numerical semigroup is pseudo-symmetric if and only if its semigroup ring is a Kunz ring.

Let us observe that $\left(\mathcal{S},\cap\right)$ is a semigroup and $\mathcal{I} = \left\{S \in \mathcal{S} \mid S \mbox{ is irreducible} \right\}$ is a system of generators of such a semigroup. In fact, it is well known (see \cite[Proposition~4.44]{springer}) that every numerical semigroup can be expressed as the intersection of finitely many irreducible numerical semigroups. By this reason, it is interesting to have an algorithm to compute $\mathcal{I}$. Thus, in \cite{forum} it was given an algorithm process which allows us to compute $\mathcal{I}(F) = \mathcal{I} \cap \mathcal{S}(F)$, which is obviously a finite set.

Following the notation of \cite{belga}, we say that $S\in \mathcal{S}(F)$ is an \emph{atom of $\mathcal{S}(F)$} if it cannot be expressed as the intersection of two elements of $\mathcal{S}(F)$ containing it properly. Thus, a numerical semigroup $S$ is an \emph{atomic numerical semigroup} if it is an atom of $\mathcal{S}\big(\mathrm{F}(S)\big)$. Let us observe that, if $S \in \mathcal{S}(F)$ is an irreducible numerical semigroup, then $S$ is atomic, that is, $S$ is an atom of $\mathcal{S}(F)$. However, if $S$ is an atomic numerical semigroup, then $S$ is not necessarily irreducible (see \cite[Example~26]{belga}).
 
It is clear that, if $F$ is a positive integer, then $\left(\mathcal{S}(F),\cap\right)$	is a semigroup. Moreover, the set $\mathcal{A}(F) = \left\{ S\in \mathcal{S}(F) \mid S \mbox{ is atomic} \right\}$ is a system of generators of such a semigroup. Therefore, as a generalization of the algorithm given in \cite{forum}, it would be interesting to  show an algorithmic process in order to determine, from a positive integer $F$, the set $\mathcal{A}(F)$. That is precisely the main purpose of this work. By the way, such as was observed in \cite{belga}, $\mathcal{I}(F)$ is not a system of generators of $\left(\mathcal{S}(F),\cap\right)$.

\section{ANI-semigroups}\label{irreducible-ani}

We propose to characterize the atomic numerical semigroups which are not irreducible, that is, the so-called \emph{$\mathrm{ANI}$-semigroups} (see \cite{linear}).

Let $S$ be a numerical semigroup. We denote by
$$\mathrm{E}(S) = \big\{x \in \mathbb{N} \setminus S \mid 2x\in S \mbox{ and } x+s\in S \mbox{ for all } s\in S\setminus \{0\} \big\}.$$ 
The following result has an easy proof and shows an interpretation of the elements of $\mathrm{E}(S)$, which are called \emph{special gaps} of $S$ (see \cite{springer}).

\begin{lemma}\label{lem1}
	Let $S$ be a numerical semigroup and let $x\in \mathbb{N} \setminus S$. Then $x\in \mathrm{E}(S)$ if and only if $S\cup \{x\}$ is a numerical semigroup.
\end{lemma}

Let $X$ be a set. We denote by $\#X$ the cardinality of $X$. Let us observe that $\#\mathrm{E} (\mathbb{N})=0$ and that $\#\mathrm{E}(S)\geq1$ for all numerical semigroup $S\not= \mathbb{N}$ (since $\mathrm{F}(S) \in \mathrm{E}(S)$).

The next result is a reformulation of Lemma~2 and Proposition~5 in \cite{linear}.

\begin{proposition}\label{prop2}
	Let $S$ be a numerical semigroup. Then $S$ is atomic if and only if $\#\mathrm{E}(S)\leq2$. Moreover, $S$ is irreducible if and only if $\#\mathrm{E}(S) \in \{0,1\}$.
\end{proposition}

Let us recall that, if $F$ is a positive integer, then $\mathcal{I}(F)=\mathcal{I} \cap \mathcal{S}(F)$. In addition, let us denote by $\mathcal{N}(F)$ the set of all $\mathrm{ANI}$-semigroups with Frobenius number equal to $F$. Then, from Proposition~\ref{prop2}, we have that $\{\mathcal{I}(F),\mathcal{N}(F)\}$ is a partition of $\mathcal{A}(F)$. On the other hand, in \cite{forum} it is shown an algorithmic process which allows us to compute $\mathcal{I}(F)$. Therefore, from now on, we will focus our attention on giving an algorithmic process in order to compute $\mathcal{N}(F)$.

The following result is an immediate consequence of Proposition~\ref{prop2}.

\begin{corollary}\label{cor3}
	A numerical semigroup $S$ is an $\mathrm{ANI}$-semigroup if and only if $\#\mathrm{E}(S)=2$.
\end{corollary}

In the next result, which is Theorem~19 of \cite{linear}, are exposed the conditions that must fulfil two positive integers $g_1,g_2$ in order to exists at least one $\mathrm {ANI}$-semigroup $S$ such that $\mathrm{E}(S)=\{g_1,g_2\}$.

\begin{proposition}\label{prop4}
	Let $g_1,g_2$ be two positive integers such that $g_1<g_2$. Then there exists a numerical semigroup $S$ such that $\mathrm{E}(S)=\{g_1,g_2\}$ if and only if $\frac{g_2}{2}<g_1$ and it is fulfil at least one of the followings two conditions:
	\begin{itemize}
		\item $g_2-g_1 \mid g_2$;
		\item $2g_1-g_2 \nmid g_2$.
	\end{itemize}
\end{proposition}

As a consequence of the above proposition, we give the first step to obtain the algorithmic process which compute $\mathcal{N}(F)$. In fact, we begin considering the set
$$\mathrm{L}(F)=\left\{ l \in \mathbb{N} \;\Big\vert\; \frac{F}{2}<l<F \mbox{ and } \big(F-l\mid F \mbox{ or } 2l-F\nmid F\big) \right\}.$$
Now, for each $l \in \mathrm{L}(F)$, we compute the set
$$\mathcal{N}(F,l) = \big\{ S \mid S \mbox{ is a numerical semigroup with } \mathrm{E}(S)=\{l,F\} \big\}.$$
Then it is clear that $\mathcal{N}(F)=\bigcup_{l\in \mathrm{L}(F)} \mathcal{N}(F,l)$.

Let $S$ be a numerical semigroup. We denote by
$$\mathrm{PF}(S)=\left\{ x\in \mathbb{Z} \setminus S \mid x+s\in S \mbox{ for all } s\in S \setminus \{0\} \right\}.$$
This set was introduced in \cite{froberg} and its cardinality is an important invariant of $S$, the so-called \emph{type} of $S$. Let us observe that $\mathrm{E}(S)=\left\{ x\in\mathrm{PF}(S) \mid 2x \in S \right\}$. We denote by $\mathrm{BPF}(S)=\left\{ x\in \mathrm{PF}(S) \;\big\vert\; x>\frac{\mathrm{F}(S)}{2} \right\}$.

\begin{lemma}\label{lem5}
	Let $g_1,g_2$ be two positive integers such that $\frac{g_2}{2}<g_1<g_2$. If $S$ is a numerical semigroup and $\mathrm{E}(S)=\{g_1,g_2\}$, then $\mathrm{BPF}(S)=\{g_1,g_2\}$.
\end{lemma}

\begin{proof}
	If $x\in \mathrm{BPF}(S)$, then $2x>\mathrm{F}(S)$. Therefore, $2x\in S$ and, consequently, $x\in \mathrm{E}(S)$. Thus, we have that $\mathrm{BPF}(S)\subseteq \mathrm{E}(S)$.
	
	On the other hand, since $\mathrm{E}(S) \subseteq \mathrm{PF}(S)$ and $\frac{g_2}{2}<g_1<g_2$, we easily deduce that $\mathrm{E}(S)\subseteq \mathrm{BPF}(S)$.
\end{proof}

From the above lemma and Theorem~8 in \cite{belga}, we get the following result.

\begin{proposition}\label{prop6}
	Let $g_1,g_2$ be two positive integers such that $\frac{g_2}{2}<g_1<g_2$ and let $S$ be a numerical semigroup such that $\mathrm{E}(S)=\{g_1,g_2\}$. Then there exist $S_1\in \mathcal{I}(g_1)$ and $S_2\in \mathcal{I}(g_2)$ such that $S=S_1\cap S_2$.
\end{proposition}

If $h_1,\ldots,h_n$ are positive integers, then
\begin{itemize}
	\item $\mathcal{\overline S}(h_1,\ldots,h_n)$ denotes the set of all numerical semigroups $S$ such that $\{h_1,\ldots,h_n\}\cap S=\emptyset$;
	\item $\mathcal{M}(h_1,\ldots,h_n)$ is the set of the maximal elements (with respect to the inclusion order) of $\mathcal{\overline S}(h_1,\ldots,h_n)$.
\end{itemize}

The next result is Proposition~12 of \cite{linear}.

\begin{proposition}\label{prop7}
	Let $S$ be a numerical semigroup and let $g_1,g_2$ be positive integers such that $\frac{g_2}{2}<g_1<g_2$. Then the following conditions are equivalent.
	\begin{enumerate}
		\item $\mathrm{E}(S)=\{g_1,g_2\}$.
		\item $S\in \mathcal{M}(g_1,g_2)$ and $g_2-g_1\notin S$.
	\end{enumerate}
\end{proposition}

The following result is an improvement of Proposition~\ref{prop6}.

\begin{proposition}\label{prop8}
	Let $g_1,g_2$ be two positive integers such that $\frac{g_2}{2}<g_1<g_2$ and let $S$ be a numerical semigroup such that $\mathrm{E}(S)=\{g_1,g_2\}$. Then there exist $S_1\in \mathcal{I}(g_1)$ and $S_2\in \mathcal{I}(g_2)$ such that $g_1\in S_2$ and $S=S_1\cap S_2$.
\end{proposition}

\begin{proof}
	From Proposition~\ref{prop6}, there exist $S_1\in \mathcal{I}(g_1)$ and $S_2\in \mathcal{I}(g_2)$ such that $S=S_1\cap S_2$.
	
	Now, let us suppose that $g_1\notin S_2$. In such a case, $S_2\in \mathcal{\overline S}(g_1,g_2)$. Then, by Proposition~\ref{prop7}, we have that $S=S_2$ (since $S=S_1\cap S_2 \subseteq S_2$) and, by Proposition~\ref{prop2}, that $\#\mathrm{E}(S)\leq1$. Thus, we get a contradiction and, in consequence, $g_1\in S_2$.
\end{proof}

The next result is Lemma~5 of \cite{computer}.

\begin{lemma}\label{lem9}
	If $S\in\mathcal{S}(F)$, then $S\cup \big\{ x\in\mathbb{N}\setminus S \;\big\vert\; F-x \notin S \mbox{ and } x>\frac{F}{2} \big\} \in \mathcal{I}(F)$.
\end{lemma}

At this point, we are ready to prove the main result of this section.

\begin{theorem}\label{thm10}
	Let $S$ be a numerical semigroup and let $g_1,g_2$ be positive integers such that $\frac{g_2}{2}<g_1<g_2$. Then the following conditions are equivalent.
	\begin{enumerate}
		\item $\mathrm{E}(S)=\{g_1,g_2\}$.
		\item $S$ is a maximal element (with respect to the inclusion order) of the set $\mathcal{I}_\cap(g_1,g_2)=\big\{ S_1\cap S_2 \mid S_1\in\mathcal{I}(g_1), \; S_2\in\mathcal{I}(g_2) \mbox{ and } g_1\in S_2 \big\}$ and, in addition, $g_1$ belongs to every element of $\mathcal{I}(g_2)$ which contains $S$.
	\end{enumerate}
\end{theorem}

\begin{proof}
	$\mathit{1.\Rightarrow 2.)}$ By Proposition~\ref{prop8}, we know that, if $\mathrm{E}(S)=\{g_1,g_2\}$, then $S\in \mathcal{I}_\cap(g_1,g_2)$. Now, if $S$ is not a maximal element of $\mathcal{I}_\cap(g_1,g_2)$, then we deduce that $S\notin \mathcal{M}(g_1,g_2)$ and, by Proposition~\ref{prop7}, that $\mathrm{E}(S)\not=\{g_1,g_2\}$. That is, we have a contradiction.
	
	On the other hand, let $S^* \in \mathcal{I}(g_2)$ be a numerical semigroup such that $S\subseteq S^*$ and $g_1 \notin S^*$. Then $S^* \in \mathcal{\overline S}(g_1,g_2)$ and $S \subsetneq S^*$. Therefore, $S\notin \mathcal{M}(g_1,g_2)$ and, by applying again Proposition~\ref{prop7}, we conclude that $\mathrm{E}(S)\not=\{g_1,g_2\}$. That is, the same contradiction that in the above paragraph.
	
	$\mathit{2.\Rightarrow 1.)}$ Let $S$ be a numerical semigroup fulfils the conditions of $\mathit{2}$. From Proposition~\ref{prop7}, in order to see that $\mathrm{E}(S)=\{g_1,g_2\}$, it is enough to show that $S\in \mathcal{M}(g_1,g_2)$ and $g_2-g_1 \in S$.
	
	First of all, since $S\in \mathcal{I}_\cap(g_1,g_2)$, then there exist $S_1\in\mathcal{I}(g_1)$ and $S_2\in\mathcal{I}(g_2)$ such that $g_1\in S_2$ and $S=S_1 \cup S_2$. Because $g_2\notin S_2$ and $g_1 \in S_2$, we deduce that $g_2-g_1 \notin S_2$ and, therefore, $g_2-g_1 \notin S$.
	
	Now, let us suppose that $S\notin \mathcal{M}(g_1,g_2)$. Then, there exists $T\in \mathcal{\overline S}(g_1,g_2)$ such that $S\subsetneq T$. By applying Lemma~\ref{lem9}, we easily deduce that there exist $T_1\in\mathcal{I}(g_1)$ and $T_2\in\mathcal{I}(g_2)$ such that $g_1\in T_2$ and $T=T_1\cap T_2$ (observe that $T_2\in\mathcal{I}(g_2)$, $S\subset T_2$ and, consequently, $g_1\in T_2$). Therefore, $T\in \mathcal{I}_\cap(g_1,g_2)$ and $S\subsetneq T$, in contradiction with the maximality of $S$ in $\mathcal{I}_\cap(g_1,g_2)$.
\end{proof}

\section{The algorithm}\label{algorithm}

From here on, we are going to focus on showing an algorithm to compute, from two given positive integers $g_1$ and $g_2$ (in the conditions of Proposition~\ref{prop4}), the maximal elements of the set $\mathcal{I}_\cap(g_1,g_2)$.

We begin recalling briefly the algorithmic method described in \cite{forum} to compute $\mathcal{I}(F)$. First of all, we need to introduce several concepts and results.

If $A$ is a non-empty subset of $\mathbb{N}$, we denote by $\langle A \rangle$ the submonoid of $(\mathbb{N},+)$ generated by $A$, that is,
$$\langle A \rangle = \big\{ \lambda_1a_1+\cdots+\lambda_na_n \mid n\in \mathbb{N} \setminus\{0\}, \; a_1,\ldots,a_n\in A, \; \lambda_1,\ldots,\lambda_n \in \mathbb{N} \big\}.$$
It is well known (see for instance \cite{springer}) that $\langle A \rangle$ is a numerical semigroup if and only if $\gcd(A)=1$. If $S$ is a numerical semigroup and $S=\langle A \rangle$, then we say that $A$ is a \emph{system of generators} of $S$. In addition, if no proper subset of $A$ generates $S$, then we say that $A$ is a \emph{minimal system of generators} of $S$. It is also well known (see \cite{springer}) that every numerical semigroup has a unique minimal system of generators, which is denoted by $\mathrm{msg}(S)$. The smallest minimal generator of $S$ is called the \emph{multiplicity} of $S$ and denoted by $\mathrm{m}(S)$.

The next result is Proposition~2.7 of \cite{forum}. \big(Let us recall that, if $m_1,\ldots,m_s\in \mathbb{N}$ and $0\leq m_1\leq\cdots\leq m_s$, then $\left\{ 0,m_1,\dots,m_s,\to \right\}=\left\{ 0,m_1,\dots,m_s\right\} \cup \left\{m\in \mathbb{N} \mid m> m_s \right\}$.\big)

\begin{proposition}\label{prop11}
	Let $F$ be a positive integer. Then there exists a unique irreducible numerical semigroup $\mathrm{C}(F)$ with Frobenius number $F$ and all its minimal generators greater that $\frac{F}{2}$.
	Moreover,
	$$\mathrm{C}(F) =\left\{ \begin{array}{ll}
	\big\{ 0,\frac{F+1}{2},\to \big\} \setminus \{F\}, & \mbox{if $F$ is odd}, \\[3pt]
		\big\{ 0,\frac{F}{2}+1,\to \big\} \setminus \{F\}, & \mbox{if $F$ is even}.
	\end{array} \right.$$
\end{proposition}

Let us define the graph $\mathrm{G}(\mathcal{I}(F))$ as follow: $\mathcal{I}(F)$ is the set of vertices and $(T,S)\in \mathcal{I}(F) \times \mathcal{I}(F)$ is an edge if $\mathrm{m}(T)<\frac{F}{2}$ and $S=(T\setminus \{\mathrm{m}(T) \}) \cup \{F-\mathrm{m}(T) \}$.

The next result is Theorem~2.9 of \cite{forum} and the base of the algorithmic process described in \cite{forum} to compute $\mathcal{I}(F)$. \big(Let us recall that a graph $G$ is a \emph{tree} if there exists a vertex $r$ (the so-called \emph{root} of $G$) such that, for any other vertex $v$ of $G$, there exists a unique path connecting $v$ and $r$. In addition, if $(v,w)$ is an edge of the tree, then we say that $v$ is a \emph{child} of $w$.\big)

\begin{theorem}\label{thm12}
	Let $F$ be a positive integer. Then the graph $\mathrm{G}(\mathcal{I}(F))$ is a tree with root $\mathrm{C}(F)$. Moreover, if $S\in \mathcal{I}(F)$, then the set formed by the children of $S$ in $\mathrm{G}(\mathcal{I}(F))$ is
	$$\left\{ \big(S\setminus \{x\}\big) \cup \{F-x\} \;\Bigg\vert\; \begin{array}{c}
	x\in\mathrm{msg}(S), \;\frac{F}{2}<x<F, \; 2x-F\notin S, \\[3pt] 3x\not=2F, \; 4x\not=3F \mbox{ and }\, F-x<\mathrm{m}(S)	\end{array}  \right\}.$$
\end{theorem}

The above theorem allows us to recurrently build, beginning from the root and adding the children of the known vertices, the graph $\mathrm{G}(\mathcal{I}(F))$ and, consequently, the set $\mathcal{I}(F)$. The following one is the Example~2.10 of \cite{forum}, in which it is computed $\mathrm{G}(\mathcal{I}(11))$.

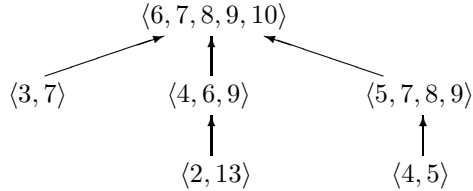
\begin{example}\label{exmpl13}
	\mbox{ } \smallskip
	
		\begin{center}
		\begin{picture}(190,65)
		\put(58,62){$\langle 6,7,8,9,10 \rangle$}
		\put(22,42){\vector(3,1){45}} \put(85,42){\vector(0,1){15}} \put(150,42){\vector(-3,1){45}}
		\put(8,32){$\langle 3,7 \rangle$}  \put(68,32){$\langle 4,6,9 \rangle$} \put(143,32){$\langle 5,7,8,9 \rangle$}
		\put(85,12){\vector(0,1){15}} \put(165,12){\vector(0,1){15}}
		\put(73,2){$\langle 2,13 \rangle$} \put(153,2){$\langle 4,5 \rangle$}
		\end{picture}
	\end{center}
\end{example}

In Section~3 of \cite{forum} it is shown how we can build the tree $\mathrm{G}(\mathcal{I}(F))$ using the so-called \emph{Kunz-coordinates vector} of a numerical semigroup. In essence, the idea of this type of coordinates is that each element $S\in\mathcal{S}(F)$ is expressed by an $F$-tupla $(x_1,\ldots,x_F)\in\{0,1\}^F$, where $x_i=0$ if and only if $i\in S$. In this way, given the $F$-tupla associated to a numerical semigroup $S\in\mathcal{S}(F)$, in Section~3 of \cite{forum} it is explained how we can build the $F$-tuplas associated to the children of $S$ in the tree $\mathrm{G}(\mathcal{I}(F))$. For instance, as continuation of Example~\ref{exmpl13}, the next one is Example~3.6 of \cite{forum}, where $\mathrm{G}(\mathcal{I}(11))$ is computed by depicting each element of $\mathcal{I}(11)$ by an $11$-tupla.
	
\begin{example}\label{exmpl14}
	\mbox{ } \smallskip
	
	\begin{center}
		\begin{picture}(340,65)
		\put(120,62){$(1,1,1,1,1,0,0,0,0,0,1)$}
		\put(105,42){\vector(2,1){30}} \put(173.5,42){\vector(0,1){15}} \put(243,42){\vector(-2,1){30}}
		\put(8,32){$(1,1,0,1,1,0,0,1,0,0,1)$}  \put(120,32){$(1,1,1,0,1,0,1,0,0,0,1)$} \put(232,32){$(1,1,1,1,0,1,0,0,0,0,1)$}
		\put(173.5,12){\vector(0,1){15}} \put(285.5,12){\vector(0,1){15}}
		\put(120,2){$(1,0,1,0,1,0,1,0,1,0,1)$} \put(232,2){$(1,1,1,0,0,1,1,0,0,0,1)$}
		\end{picture}
	\end{center}
\end{example}

The following algorithm allows us to build, from two given positive integers $g_1$ and $g_2$ (in the conditions of Proposition~\ref{prop4}), the set formed by all numerical semigroups $S$ such that $\mathrm{E}(S)=\{g_1,g_2\}$.

If $(a_1,\ldots,a_n),(b_1,\ldots,b_n)\in \mathbb{N}^n$, then we denote by
$$(a_1,\ldots,a_n)\vee (b_1,\ldots,b_n) =\big(\max\{a_1,b_1\},\ldots,\max\{a_n,b_n\} \big).$$
Moreover, if $C \subseteq {\mathbb N}^n$, then we denote by $\mathrm{minimals}(C)$ the set of minimals elements in $C$ with respect the usual component-wise order. 

\begin{algorithm}\label{alg15}
	\mbox{ }
	
	\noindent\textit{INPUT}: $g_1$, $g_2$ (positive integers fulfilling the conditions of Proposition~\ref{prop4}).
	
	\noindent\textit{OUTPUT}: The set $\big\{ S \mid S \mbox{ is a numerical semigroup and } \mathrm{E}(S)=\{g_1,g_2\} \big\}$.
	
	\vspace{-5pt}
	\begin{enumerate}
		\item[\textit{1.}] Compute $\mathcal{I}(g_1)$ using the algorithmic method shown in \cite{forum} and the Kunz-coordinates vector. \big(Thus, we express each element of $\mathcal{I}(g_1)$ by an element of $\{0,1\}^{g_1}$.\big)
		\vspace{-3pt}
		\item[\textit{2.}] Compute $\mathcal{I}(g_2)$ using the algorithmic method shown in \cite{forum} and the Kunz-coordinates vector. \big(Thus, we express each element of $\mathcal{I}(g_2)$ by an element of $\{0,1\}^{g_2}$.\big)
		\vspace{-3pt}
		\item[\textit{3.}] Set $A=\big\{ (x_1,\ldots,x_{g_1},0,\ldots,0)\in\{0,1\}^{g_2} \;\big\vert\; (x_1,\ldots,x_{g_1})\in \mathcal{I}(g_1) \big\}$, \par 
		$B_0=\big\{ (x_1,\ldots,x_{g_2})\in\mathcal{I}(g_2) \;\big\vert\; x_{g_1}=0 \big\}$, \par and $B_1=\big\{ (x_1,\ldots,x_{g_2})\in\mathcal{I}(g_2) \;\big\vert\; x_{g_1}=1 \big\}$.
		\vspace{-3pt}
		\item[\textit{4.}] Compute $C=\mathrm{minimals}\left\{a\vee b \mid a\in A \mbox{ and } b\in B_0 \right\}$.
		\vspace{-3pt}
		\item[\textit{5.}] Return $D=\left\{c\in C \mid b \vee c \not= c \mbox{ for all } b \in B_1 \right\}$.
	\end{enumerate}
\end{algorithm}

In the next example we see how apply the above algorithm.

\begin{example}\label{exmpl15}
	Let us compute all numerical semigroups $S$ with $\mathrm{E}(S)=\{8,11\}$. First of all, let us observe that, by applying Proposition~\ref{prop4}, such numerical semigroups exist. Now, we have that (see Example~\ref{exmpl14})
	\begin{itemize}
		\item $\mathcal{I}(8)=\left\{(1,1,1,1,0,0,0,1),(1,1,0,1,1,0,0,1) \right\}$.
		\item $\mathcal{I}(11)=\left\{(1,1,1,1,1,0,0,0,0,0,1),(1,1,0,1,1,0,0,1,0,0,1),\right.$\par
		$\hspace{1.45cm}\left.(1,1,1,0,1,0,1,0,0,0,1),(1,1,1,1,0,1,0,0,0,0,1),\right.$\par
		$\hspace{1.45cm}\left.(1,0,1,0,1,0,1,0,1,0,1),(1,1,1,0,0,1,1,0,0,0,1)\right\}$.
	\end{itemize}
	Thereby,
	\begin{itemize}
		\item $A=\{a_1,a_2\} = \left\{(1,1,1,1,0,0,0,1,0,0,0),(1,1,0,1,1,0,0,1,0,0,0) \right\}$.
		\item $B_0=\{b_1,b_3,b_4,b_5,b_6\} = \left\{(1,1,1,1,1,0,0,0,0,0,1),\right.$\par
		$\hspace{2.5cm}\left.(1,1,1,0,1,0,1,0,0,0,1),(1,1,1,1,0,1,0,0,0,0,1),\right.$\par
		$\hspace{2.5cm}\left.(1,0,1,0,1,0,1,0,1,0,1),(1,1,1,0,0,1,1,0,0,0,1)\right\}$.
		\item $B_1=\{b_2\} = \left\{(1,1,0,1,1,0,0,1,0,0,1)\right\}$.
	\end{itemize}
	Thus,
	\begin{itemize}
		\item $C=\mathrm{minimals}\left\{(1,1,1,1,1,0,0,1,0,0,1),(1,1,1,1,1,0,1,1,0,0,1),\right.$\par
		$\hspace{2.35cm}\left.(1,1,1,1,0,1,0,1,0,0,1),(1,1,1,1,1,0,1,1,1,0,1),\right.$\par
		$\hspace{2.35cm}\left.(1,1,1,1,0,1,1,1,0,0,1),(1,1,1,1,1,1,0,1,0,0,1),\right.$\par
		$\hspace{2.35cm}\left.(1,1,1,1,1,1,1,1,0,0,1)\right\} \Rightarrow$\par
		$C=\{c_1,c_2\}=\left\{(1,1,1,1,1,0,0,1,0,0,1),(1,1,1,1,0,1,0,1,0,0,1)\right\}$.
	\end{itemize}
	Lastly, since $b_2\leq c_1$ (and, consequently, $b_2\vee c_1=c_1$), we can conclude that $D=\left\{(1,1,1,1,0,1,0,1,0,0,1)\right\}$. Therefore, $\{0,5,7,9,10,12,\to\}=\langle5,7,9,13\rangle$ is the unique numerical semigroup $S$ with $\mathrm{E}(S)=\{8,11\}$.
	
\end{example}

Finally, fixed a positive integer $F$, the next algorithm allows us to compute the set $\mathcal{A}(F)$ formed by all atomic numerical semigroups with Frobenius number equal to $F$.

\begin{algorithm}\label{alg16}
	\mbox{ }
	
	\noindent\textit{INPUT}: a positive integer $F$.
	
	\noindent\textit{OUTPUT}: $\mathcal{A}(F)$.
	
	\vspace{-5pt}
	\begin{enumerate}
		\item[\textit{1.}] Compute $\mathcal{I}(F)$ using the Kunz-coordinates vector.
		\vspace{-3pt}
		\item[\textit{2.}] Compute $\mathrm{L}(F)=\left\{ l \in \mathbb{N} \;\big\vert\; \frac{F}{2}<l<F \mbox{ and } \big(F-l\mid F \mbox{ or } 2l-F\nmid F\big) \right\}$.
		\vspace{-3pt}
		\item[\textit{3.}] For each $l\in\mathrm{L}(F)$, and using Algorithm~\ref{alg15}, compute the set $\mathcal{N}(F,l) = \big\{ S \mid S \mbox{ is a numerical semigroup with } \mathrm{E}(S)=\{l,F\} \big\}$.
		\vspace{-3pt}
		\item[\textit{4.}] Return $\mathcal{I}(F) \cup \mathcal{N}(F) = \mathcal{I}(F) \cup \Big(\bigcup_{l\in \mathrm{L}(F)} \mathcal{N}(F,l)\Big)$.
	\end{enumerate}
\end{algorithm}

Let us see how this algorithm works in practice.

\begin{example}\label{exmpl16}
	To compute $\mathcal{A}(11)$, first we observe that, from Example~\ref{exmpl14},
	\begin{itemize}
		\item $\mathcal{I}(11)=\left\{(1,1,1,1,1,0,0,0,0,0,1),(1,1,0,1,1,0,0,1,0,0,1),\right.$\par
		$\hspace{1.45cm}\left.(1,1,1,0,1,0,1,0,0,0,1),(1,1,1,1,0,1,0,0,0,0,1),\right.$\par
		$\hspace{1.45cm}\left.(1,0,1,0,1,0,1,0,1,0,1),(1,1,1,0,0,1,1,0,0,0,1)\right\}$.
	\end{itemize}
	Now, from Proposition~\ref{prop4}, we can assert that $\mathrm{L}(11)=\left\{7,8,9,10\right\}$ is the set of positive integers $n$ such that there exists a numerical semigroup $S$ such that $\mathrm{E}(S)=\{n,11\}$. Then, by applying Algorithm~\ref{alg15}, we have that
	\begin{itemize}
		\item $\mathcal{N}(11,7) = \left\{ (1,1,0,1,1,0,1,1,0,0,1) \right\}$,
		\item $\mathcal{N}(11,8) = \left\{ (1,1,1,1,0,1,0,1,0,0,1) \right\}$,
		\item $\mathcal{N}(11,9) = \left\{ (1,1,1,1,1,0,0,0,1,0,1), (1,1,1,1,0,1,0,0,1,0,1) \right\}$,
		\item $\mathcal{N}(11,10) = \left\{ (1,1,1,1,1,0,0,0,0,1,1), (1,1,1,0,1,1,1,0,0,1,1),\right.$\par
		$\hspace{2.1cm}\left. (1,1,0,1,1,0,1,1,0,1,1) \right\}$.
	\end{itemize}
	Therefore, by using the minimal system of generators for a numerical semigroup,
	\begin{itemize}
		\item $\mathcal{I}(11) = \left\{\langle 6,7,8,9,10 \rangle, \langle 3,7 \rangle, \langle 4,6,9 \rangle, \langle 5,7,8,9 \rangle, \langle 2,13 \rangle, \langle 4,5 \rangle \right\}$,
		\item $\mathcal{N}(11) = \left\{\langle 3,10,14 \rangle, \langle 5,7,9,13 \rangle, \langle 6,7,8,10 \rangle, \langle 5,7,8 \rangle, \right.$\par
		$\hspace{1.6cm}\left. \langle 6,7,8,9 \rangle, \langle 4,9,14,15 \rangle, \langle 3,13,14 \rangle \right\}$.
	\end{itemize}
 	And finally,
 		$$\mathcal{A}(11) = \left\{\langle 6,7,8,9,10 \rangle, \langle 3,7 \rangle, \langle 4,6,9 \rangle, \langle 5,7,8,9 \rangle, \langle 2,13 \rangle, \langle 4,5 \rangle, \langle 3,10,14 \rangle,  \right.$$
 		$$\hspace{1.9cm}\left. \langle 5,7,9,13 \rangle, \langle 6,7,8,10 \rangle, \langle 5,7,8 \rangle,\langle 6,7,8,9 \rangle, \langle 4,9,14,15 \rangle, \langle 3,13,14 \rangle \right\}\!.$$
\end{example}



\begin{thebibliography}{00}
	
	\bibitem{barucci} V. Barucci, D.~E. Dobbs, and M. Fontana, \emph{Maximality Properties in Numerical Semigroups and Applications to One-Dimensional Analytically Irreducible Local Domains}, Mem. Amer. Math. Soc. \textbf{598} (1997). 
	
	\bibitem{computer} V. Blanco, J.~C. Rosales, On the enumeration of the set of numerical semigroups with fixed Frobenius number, \emph{Comput. Math. Appl.} \textbf{63} (2012), 1204--1211.
	
	\bibitem{forum} V. Blanco, J.~C. Rosales, The tree of irreducible numerical semigroups with fixed Frobenius number, \emph{Forum Math.} \textbf{25} (2013), 1249--1261.
	
	\bibitem{froberg} R. Fr\"oberg, G. Gottlieb and R. H\"aggkvist, On numerical semigroups, \emph{Semigroup Forum} \textbf{35} (1987), 63--83.
	
	\bibitem{kunz} E. Kunz, The value-semigroup of a one-dimensional Gorenstein ring, \emph{Proc. Amer. Math. Soc.} \textbf{25} (1970), 748--751.
	
	\bibitem{linear} J.~C. Rosales, Atoms of the set of numerical semigroups with fixed Frobenius number, \emph{Linear Algebra Appl.} \textbf{430} (2009), 41--51.
	
	\bibitem{belga} J.~C. Rosales, M.~B. Branco, Decomposition of a numerical semigroup as an intersection of irreducible numerical semigroups, \emph{Bull. Belg. Math. Soc. Simon Stevin} \textbf{9} (2002), 373--381.
	
	\bibitem{pacific} J.~C. Rosales, M.~B. Branco, Irreducible numerical semigroups, \emph{Pacific J. Math.} \textbf{209(1)} (2003), 131--143.	
	
	\bibitem{springer} J.~C. Rosales and P.~A. Garc\'{\i}a-S\'anchez, \emph{Numerical semigroups} (Developments in Mathematics, vol. \textbf{20}, Springer, New York, 2009).
	
\end{thebibliography}
\end{document}